\documentclass[11pt,a4paper]{article}
\usepackage[utf8]{inputenc}
\usepackage{lmodern}
\usepackage[T1]{fontenc}
\usepackage[english]{babel}
\usepackage{ifpdf}
\usepackage[left=1in, right=1in, top=1in, bottom=1in]{geometry}
\usepackage[dvipsnames]{xcolor}
\usepackage[colorlinks=true,linkcolor=teal,citecolor=teal]{hyperref}

\hypersetup{
	pdftitle={},
	pdfauthor={}
} 

\usepackage{graphicx, amsmath, amsthm, amssymb, enumitem, mathrsfs} 
\usepackage{subcaption}
\usepackage{float}
\usepackage{booktabs} 

\usepackage[capitalize]{cleveref}
\usepackage[square,numbers]{natbib}
\usepackage{soul}

\usepackage{tikz}
\usepackage{tikz-3dplot} 
\usepackage{pgfplots, pgfplotstable}
\usetikzlibrary{3d} 
\usetikzlibrary{arrows.meta}
\usepackage{algorithm}
\usepackage{algpseudocode}

\newcommand{\supp}{{\mathsf{supp }\hspace{.075cm}}}

\newcommand{\prob}{{\mathscr{P}}}
\newcommand{\wass}{{\mathcal{W}}}
\newcommand{\setoftrees}{{\mathcal{T}}}
\newcommand{\stategraph}{{\mathbb{T}}}
\newcommand{\embedding}{{\mathsf{k}}}
\newcommand{\flowpoly}{{\mathcal{K}}}

\newtheorem{theorem}{Theorem}[section]

\newtheorem{remark}[theorem]{Remark}
\newtheorem{definition}[theorem]{Definition}
\newtheorem{lemma}[theorem]{Lemma}
\newtheorem{proposition}[theorem]{Proposition}

\usepackage{xcolor, soul}

\title{Optimal Transport on Graphs and Stochastically Evolving Trees}
\author{Fan Chung \and Sawyer Jack Robertson}
\date{\today}

\begin{document}
	
	\maketitle

    \begin{abstract}
        We give an effective algorithm for determining the transportation distance between two given probability density functions defined on the vertices of a graph $G=(V,E)$ by analyzing an associated \emph{polytope}. The vertices of the polytope correspond to feasible flows on spanning trees in $G$, and the $1$-skeleton of the polytope is a projection of the spanning tree state graph associated with the Glauber dynamics on $G$. The optimal value of this transportation problem, known as the $1$-Wasserstein distance, can be computed by tracing the transportation cost along the vertices of this polytope. We show that a local minimum of the transportation cost is also a global minimum, and this leads to a steepest descent algorithm for solving the transportation problem. If the probability density functions take discrete values in $\delta \mathbb{Z}$ for some $\delta>0$, then the optimal transport cost can be reached in at most $\frac{|V|-1}{\delta}$ steps. As an application, we give an efficient algorithm for computing the Ollivier--Ricci curvature of a graph.
    \end{abstract}

    \section{Introduction}
        
    Let $G=(V,E)$ be a finite, simple, and connected graph, and let $\prob(V)\subseteq\mathbb{R}^{V}$ denote the set of probability density functions on its vertex set. For fixed $\mu,\nu\in\prob(V)$, we consider the following optimization program known as the \emph{optimal transportation problem}:
        \begin{align}\label{eq:graph-transport}
            \inf_{\pi\in\Pi(\mu,\nu)} \sum_{x, y\in V} \pi(x,y)\, d(x,y).
        \end{align}
    Here, $\Pi(\mu,\nu)\subseteq\mathbb{R}^{V\times V}$ denotes the set of couplings of $\mu$ and $\nu$, and $d(\cdot,\cdot)$ is a fixed metric on $V$. We denote by $d_G(\cdot,\cdot)$ the shortest path metric on $G$, and in the case where $d=d_G$, the optimal value of the program~\cref{eq:graph-transport} is known as the \emph{$1$-Wasserstein distance} or \emph{optimal transportation distance} between $\mu$ and $\nu$ on $G$, denoted by $\wass(\mu,\nu)=\wass_G(\mu,\nu)$.
    
    The optimal transportation problem~\cref{eq:graph-transport} has been extensively studied in the literature (see, e.g.,~\cite[Ch. 6]{peyre2019computational} and references therein; see also~\cite{essid2018quadratically}) and is related to a number of well-known problems in optimization, graph theory, and computer science, including minimum cost flow programming~\cite{orlin1993faster}, shortest-path distance computation (see~\cite{dijkstra1959note} for an historical reference), effective resistance~\cite{robertson2024all}, and discrete curvatures~\cite{lin2011ricci}. Additional applications of optimal transport on graphs and its variants include the comparison and classification of structured or attributed graphs~\cite{titouan2019optimal,togninalli2019wasserstein}, as well as graph matching, network alignment, and node embedding~\cite{xu2019gromov}.

    In this paper we investigate the optimal transport problem from the perspective of \emph{evolving trees}. A \emph{spanning tree} $T\subseteq E$ is a collection of edges of $G$ that is acyclic and such that the resulting subgraph $(V,T)$ is connected (we use $T$ and the graph $(V,T)$ interchangeably). We denote by $\setoftrees=\setoftrees(G)\subseteq 2^{E}$ the set of all spanning trees of $G$. For a fixed spanning tree $T$ and probability density functions $\mu,\nu\in\prob(V)$, it is straightforward to see that $\wass_{T}(\mu,\nu)\ge \wass_{G}(\mu,\nu)$. It is known that $\wass_G(\cdot,\cdot)$ can be determined as the minimum of transportation distances restricted to spanning trees (see,~e.g.~\cite{montrucchio2022kantorovich}):
        \begin{align}\label{eq:tree-decomposition}
            \wass(\mu,\nu)=\min_{T\in\mathcal{T}(G)}\wass_T(\mu,\nu).
        \end{align}
    In this sense, the optimal transport on $G$ may be viewed as an optimization problem over the \emph{spanning trees} of $G$. For $T\in\setoftrees$ fixed, $\wass_T(\mu,\nu)$ can be computed in time linear in $|V|$ (see, e.g.,~\cref{eq:tree-flow-formula}). However, the formulation~\cref{eq:tree-decomposition} does not \emph{a priori} present a computational advantage, since by Cayley's formula, the number of labeled spanning trees of a graph can be exponential in the size of the vertex set.

    Our study begins with the initial step of equipping $\setoftrees$ with a specific graph structure and viewing the map $T\mapsto \wass_T(\mu,\nu)$ as a function defined on the nodes of an associated \emph{state graph}. We show that there exist strong geometric constraints on the landscape of this function. This leads to local steepest descent methods for solving~\cref{eq:tree-decomposition} that can be shown to converge in polynomial time for discrete measures.

    We consider the following Markov chain on $\setoftrees$ known as the \emph{Glauber dynamics}: starting from a spanning tree $T_0\in\setoftrees$, select an edge $e$ uniformly at random from $E\setminus T_0$ and add it to $T_0$. This creates a unique simple cycle from which an edge is deleted uniformly at random, thereby producing a new spanning tree that we denote by $T_1$. It can be shown that the process $(T_t)_{t\ge 0}$ is a Markov chain on $\setoftrees$ which is irreducible and reversible (see, e.g.,~\cite{alev2020improved}). In particular, the resulting transition probability matrix is symmetric, and defines an undirected graph structure on $\setoftrees$, denoted $\stategraph=\stategraph(G)$ and called the \emph{spanning tree state graph} of $G$. The nodes of $\stategraph$ are the spanning trees of $G$, and two nodes $S,T\in\setoftrees$ are linked in $\stategraph$ if they differ by a single edge exchange as described above.~\Cref{fig:flow-polytope} illustrates a small example including the base graph, the spanning tree state graph, and a linear projection of the corresponding tree polytope.

    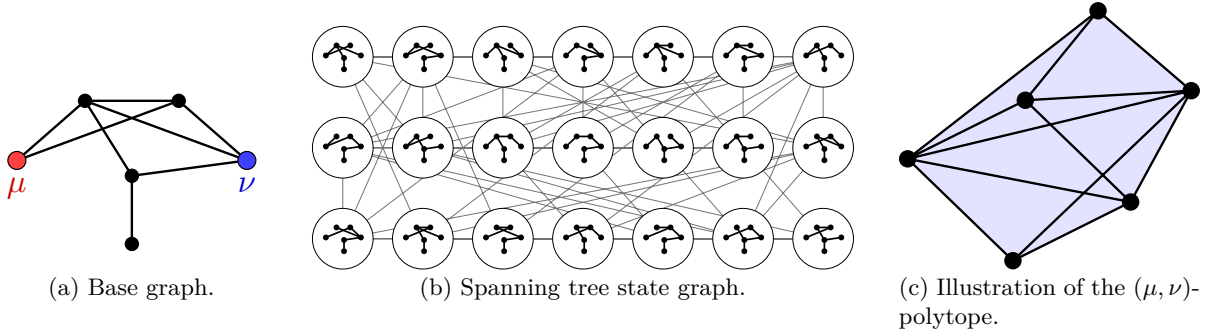
\begin{figure}[t!]
        \centering
        \captionsetup{font=small,skip=3pt}
        \captionsetup[subfigure]{font=footnotesize,skip=2pt}
        \begin{subfigure}[t]{0.225\textwidth}
            \centering
            \resizebox{\linewidth}{!}{%
                \begin{tikzpicture}[
                    x=1cm,y=1cm,
                    line cap=round,line join=round,
                    every node/.style={font=\normalfont}
                ]
                    \coordinate (s) at (-1.35, 0.10);
                    \coordinate (a) at (-0.55, 0.80);
                    \coordinate (b) at ( 0.55, 0.80);
                    \coordinate (t) at ( 1.35, 0.10);
                    \coordinate (c) at ( 0.00,-0.08);
                    \coordinate (p) at ( 0.00,-0.88);

                    \draw[black,line width=0.8pt] (s)--(a);
                    \draw[black,line width=0.8pt] (a)--(t);
                    \draw[black,line width=0.8pt] (s)--(b);
                    \draw[black,line width=0.8pt] (b)--(t);
                    \draw[black,line width=0.8pt] (a)--(b);
                    \draw[black,line width=0.8pt] (a)--(c);
                    \draw[black,line width=0.8pt] (c)--(t);
                    \draw[black,line width=0.8pt] (c)--(p);

                    \filldraw[draw=black,fill=red!75]  (s) circle (3.0pt);
                    \filldraw[draw=black,fill=blue!75] (t) circle (3.0pt);
                    \filldraw[draw=black,fill=black] (a) circle (2.2pt);
                    \filldraw[draw=black,fill=black] (b) circle (2.2pt);
                    \filldraw[draw=black,fill=black] (c) circle (2.2pt);
                    \filldraw[draw=black,fill=black] (p) circle (2.2pt);

                    \node[red!80!black,below=2pt]  at (s) {$\mu$};
                    \node[blue!80!black,below=2pt] at (t) {$\nu$};
                    \node[black!70,above=1pt] at (a) {};
                    \node[black!70,above=1pt] at (b) {};
                    \node[black!70,below right=0pt] at (c) {};
                    \node[black!70,below=1pt] at (p) {};
                \end{tikzpicture}%
            }
            \caption{Base graph.}
        \end{subfigure}
        \hfill
        \begin{subfigure}[t]{0.45\textwidth}
            \centering
            \resizebox{\linewidth}{!}{%
                \begin{tikzpicture}[
                    x=0.9cm,y=0.9cm,
                    line cap=round,line join=round,
                    treeedge/.style={black,line width=0.6pt},
                    stateedge/.style={black!55,line width=0.16pt}
                ]
                    \coordinate (n1)  at (-3.30, 1.25);
                    \coordinate (n2)  at (-2.20, 1.25);
                    \coordinate (n3)  at (-1.10, 1.25);
                    \coordinate (n4)  at ( 0.00, 1.25);
                    \coordinate (n5)  at ( 1.10, 1.25);
                    \coordinate (n6)  at ( 2.20, 1.25);
                    \coordinate (n7)  at ( 3.30, 1.25);

                    \coordinate (n8)  at (-3.30, 0.00);
                    \coordinate (n9)  at (-2.20, 0.00);
                    \coordinate (n10) at (-1.10, 0.00);
                    \coordinate (n11) at ( 0.00, 0.00);
                    \coordinate (n12) at ( 1.10, 0.00);
                    \coordinate (n13) at ( 2.20, 0.00);
                    \coordinate (n14) at ( 3.30, 0.00);

                    \coordinate (n15) at (-3.30,-1.25);
                    \coordinate (n16) at (-2.20,-1.25);
                    \coordinate (n17) at (-1.10,-1.25);
                    \coordinate (n18) at ( 0.00,-1.25);
                    \coordinate (n19) at ( 1.10,-1.25);
                    \coordinate (n20) at ( 2.20,-1.25);
                    \coordinate (n21) at ( 3.30,-1.25);

                    \draw[stateedge]
                        (n1)--(n2) (n1)--(n3) (n1)--(n5) (n1)--(n7) (n1)--(n9) (n1)--(n14) (n1)--(n16)
                        (n2)--(n4) (n2)--(n6) (n2)--(n8) (n2)--(n9) (n2)--(n15) (n2)--(n17)
                        (n3)--(n4) (n3)--(n5) (n3)--(n7) (n3)--(n10) (n3)--(n12) (n3)--(n14)
                        (n4)--(n6) (n4)--(n8) (n4)--(n11) (n4)--(n12) (n4)--(n15)
                        (n5)--(n6) (n5)--(n10) (n5)--(n13) (n5)--(n16)
                        (n6)--(n11) (n6)--(n13) (n6)--(n17)
                        (n7)--(n8) (n7)--(n9) (n7)--(n10) (n7)--(n12) (n7)--(n14) (n7)--(n18) (n7)--(n20)
                        (n8)--(n9) (n8)--(n11) (n8)--(n12) (n8)--(n15) (n8)--(n19) (n8)--(n20)
                        (n9)--(n12) (n9)--(n13) (n9)--(n20) (n9)--(n21)
                        (n10)--(n11) (n10)--(n12) (n10)--(n13) (n10)--(n18)
                        (n11)--(n12) (n11)--(n13) (n11)--(n19)
                        (n12)--(n13) (n12)--(n20)
                        (n13)--(n21)
                        (n14)--(n15) (n14)--(n16) (n14)--(n18) (n14)--(n20)
                        (n15)--(n17) (n15)--(n19) (n15)--(n20)
                        (n16)--(n17) (n16)--(n18) (n16)--(n21)
                        (n17)--(n19) (n17)--(n21)
                        (n18)--(n19) (n18)--(n20) (n18)--(n21)
                        (n19)--(n20) (n19)--(n21)
                        (n20)--(n21);

                    \begin{scope}[shift={(n1)},scale=0.86]
                        \node[circle,draw=black,fill=white,line width=0.25pt,minimum size=7.50mm,inner sep=0pt] at (0,0) {};
                        \coordinate (s) at (-0.27,0.02); \coordinate (a) at (-0.10,0.18); \coordinate (b) at (0.12,0.18);
                        \coordinate (t) at (0.29,0.02); \coordinate (c) at (0.02,-0.02); \coordinate (p) at (0.02,-0.20);
                        \draw[treeedge] (s)--(a); \draw[treeedge] (a)--(t); \draw[treeedge] (s)--(b); \draw[treeedge] (a)--(c); \draw[treeedge] (c)--(p);
                        \fill[black] (s) circle (1.20pt); \fill[black] (a) circle (1.20pt); \fill[black] (b) circle (1.20pt); \fill[black] (t) circle (1.20pt); \fill[black] (c) circle (1.20pt); \fill[black] (p) circle (1.20pt);
                    \end{scope}

                    \begin{scope}[shift={(n2)},scale=0.86]
                        \node[circle,draw=black,fill=white,line width=0.25pt,minimum size=7.50mm,inner sep=0pt] at (0,0) {};
                        \coordinate (s) at (-0.27,0.02); \coordinate (a) at (-0.10,0.18); \coordinate (b) at (0.12,0.18);
                        \coordinate (t) at (0.29,0.02); \coordinate (c) at (0.02,-0.02); \coordinate (p) at (0.02,-0.20);
                        \draw[treeedge] (s)--(a); \draw[treeedge] (a)--(t); \draw[treeedge] (s)--(b); \draw[treeedge] (c)--(t); \draw[treeedge] (c)--(p);
                        \fill[black] (s) circle (1.20pt); \fill[black] (a) circle (1.20pt); \fill[black] (b) circle (1.20pt); \fill[black] (t) circle (1.20pt); \fill[black] (c) circle (1.20pt); \fill[black] (p) circle (1.20pt);
                    \end{scope}

                    \begin{scope}[shift={(n3)},scale=0.86]
                        \node[circle,draw=black,fill=white,line width=0.25pt,minimum size=7.50mm,inner sep=0pt] at (0,0) {};
                        \coordinate (s) at (-0.27,0.02); \coordinate (a) at (-0.10,0.18); \coordinate (b) at (0.12,0.18);
                        \coordinate (t) at (0.29,0.02); \coordinate (c) at (0.02,-0.02); \coordinate (p) at (0.02,-0.20);
                        \draw[treeedge] (s)--(a); \draw[treeedge] (a)--(t); \draw[treeedge] (b)--(t); \draw[treeedge] (a)--(c); \draw[treeedge] (c)--(p);
                        \fill[black] (s) circle (1.20pt); \fill[black] (a) circle (1.20pt); \fill[black] (b) circle (1.20pt); \fill[black] (t) circle (1.20pt); \fill[black] (c) circle (1.20pt); \fill[black] (p) circle (1.20pt);
                    \end{scope}

                    \begin{scope}[shift={(n4)},scale=0.86]
                        \node[circle,draw=black,fill=white,line width=0.25pt,minimum size=7.50mm,inner sep=0pt] at (0,0) {};
                        \coordinate (s) at (-0.27,0.02); \coordinate (a) at (-0.10,0.18); \coordinate (b) at (0.12,0.18);
                        \coordinate (t) at (0.29,0.02); \coordinate (c) at (0.02,-0.02); \coordinate (p) at (0.02,-0.20);
                        \draw[treeedge] (s)--(a); \draw[treeedge] (a)--(t); \draw[treeedge] (b)--(t); \draw[treeedge] (c)--(t); \draw[treeedge] (c)--(p);
                        \fill[black] (s) circle (1.20pt); \fill[black] (a) circle (1.20pt); \fill[black] (b) circle (1.20pt); \fill[black] (t) circle (1.20pt); \fill[black] (c) circle (1.20pt); \fill[black] (p) circle (1.20pt);
                    \end{scope}

                    \begin{scope}[shift={(n5)},scale=0.86]
                        \node[circle,draw=black,fill=white,line width=0.25pt,minimum size=7.50mm,inner sep=0pt] at (0,0) {};
                        \coordinate (s) at (-0.27,0.02); \coordinate (a) at (-0.10,0.18); \coordinate (b) at (0.12,0.18);
                        \coordinate (t) at (0.29,0.02); \coordinate (c) at (0.02,-0.02); \coordinate (p) at (0.02,-0.20);
                        \draw[treeedge] (s)--(a); \draw[treeedge] (a)--(t); \draw[treeedge] (a)--(b); \draw[treeedge] (a)--(c); \draw[treeedge] (c)--(p);
                        \fill[black] (s) circle (1.20pt); \fill[black] (a) circle (1.20pt); \fill[black] (b) circle (1.20pt); \fill[black] (t) circle (1.20pt); \fill[black] (c) circle (1.20pt); \fill[black] (p) circle (1.20pt);
                    \end{scope}

                    \begin{scope}[shift={(n6)},scale=0.86]
                        \node[circle,draw=black,fill=white,line width=0.25pt,minimum size=7.50mm,inner sep=0pt] at (0,0) {};
                        \coordinate (s) at (-0.27,0.02); \coordinate (a) at (-0.10,0.18); \coordinate (b) at (0.12,0.18);
                        \coordinate (t) at (0.29,0.02); \coordinate (c) at (0.02,-0.02); \coordinate (p) at (0.02,-0.20);
                        \draw[treeedge] (s)--(a); \draw[treeedge] (a)--(t); \draw[treeedge] (a)--(b); \draw[treeedge] (c)--(t); \draw[treeedge] (c)--(p);
                        \fill[black] (s) circle (1.20pt); \fill[black] (a) circle (1.20pt); \fill[black] (b) circle (1.20pt); \fill[black] (t) circle (1.20pt); \fill[black] (c) circle (1.20pt); \fill[black] (p) circle (1.20pt);
                    \end{scope}

                    \begin{scope}[shift={(n7)},scale=0.86]
                        \node[circle,draw=black,fill=white,line width=0.25pt,minimum size=7.50mm,inner sep=0pt] at (0,0) {};
                        \coordinate (s) at (-0.27,0.02); \coordinate (a) at (-0.10,0.18); \coordinate (b) at (0.12,0.18);
                        \coordinate (t) at (0.29,0.02); \coordinate (c) at (0.02,-0.02); \coordinate (p) at (0.02,-0.20);
                        \draw[treeedge] (s)--(a); \draw[treeedge] (s)--(b); \draw[treeedge] (b)--(t); \draw[treeedge] (a)--(c); \draw[treeedge] (c)--(p);
                        \fill[black] (s) circle (1.20pt); \fill[black] (a) circle (1.20pt); \fill[black] (b) circle (1.20pt); \fill[black] (t) circle (1.20pt); \fill[black] (c) circle (1.20pt); \fill[black] (p) circle (1.20pt);
                    \end{scope}

                    \begin{scope}[shift={(n8)},scale=0.86]
                        \node[circle,draw=black,fill=white,line width=0.25pt,minimum size=7.50mm,inner sep=0pt] at (0,0) {};
                        \coordinate (s) at (-0.27,0.02); \coordinate (a) at (-0.10,0.18); \coordinate (b) at (0.12,0.18); \coordinate (t) at (0.29,0.02); \coordinate (c) at (0.02,-0.02); \coordinate (p) at (0.02,-0.20);
                        \draw[treeedge] (s)--(a); \draw[treeedge] (s)--(b); \draw[treeedge] (b)--(t); \draw[treeedge] (c)--(t); \draw[treeedge] (c)--(p);
                        \fill[black] (s) circle (1.20pt); \fill[black] (a) circle (1.20pt); \fill[black] (b) circle (1.20pt); \fill[black] (t) circle (1.20pt); \fill[black] (c) circle (1.20pt); \fill[black] (p) circle (1.20pt);
                    \end{scope}

                    \begin{scope}[shift={(n9)},scale=0.86]
                        \node[circle,draw=black,fill=white,line width=0.25pt,minimum size=7.50mm,inner sep=0pt] at (0,0) {};
                        \coordinate (s) at (-0.27,0.02); \coordinate (a) at (-0.10,0.18); \coordinate (b) at (0.12,0.18); \coordinate (t) at (0.29,0.02); \coordinate (c) at (0.02,-0.02); \coordinate (p) at (0.02,-0.20);
                        \draw[treeedge] (s)--(a); \draw[treeedge] (s)--(b); \draw[treeedge] (a)--(c); \draw[treeedge] (c)--(t); \draw[treeedge] (c)--(p);
                        \fill[black] (s) circle (1.20pt); \fill[black] (a) circle (1.20pt); \fill[black] (b) circle (1.20pt); \fill[black] (t) circle (1.20pt); \fill[black] (c) circle (1.20pt); \fill[black] (p) circle (1.20pt);
                    \end{scope}

                    \begin{scope}[shift={(n10)},scale=0.86]
                        \node[circle,draw=black,fill=white,line width=0.25pt,minimum size=7.50mm,inner sep=0pt] at (0,0) {};
                        \coordinate (s) at (-0.27,0.02); \coordinate (a) at (-0.10,0.18); \coordinate (b) at (0.12,0.18); \coordinate (t) at (0.29,0.02); \coordinate (c) at (0.02,-0.02); \coordinate (p) at (0.02,-0.20);
                        \draw[treeedge] (s)--(a); \draw[treeedge] (b)--(t); \draw[treeedge] (a)--(b); \draw[treeedge] (a)--(c); \draw[treeedge] (c)--(p);
                        \fill[black] (s) circle (1.20pt); \fill[black] (a) circle (1.20pt); \fill[black] (b) circle (1.20pt); \fill[black] (t) circle (1.20pt); \fill[black] (c) circle (1.20pt); \fill[black] (p) circle (1.20pt);
                    \end{scope}

                    \begin{scope}[shift={(n11)},scale=0.86]
                        \node[circle,draw=black,fill=white,line width=0.25pt,minimum size=7.50mm,inner sep=0pt] at (0,0) {};
                        \coordinate (s) at (-0.27,0.02); \coordinate (a) at (-0.10,0.18); \coordinate (b) at (0.12,0.18); \coordinate (t) at (0.29,0.02); \coordinate (c) at (0.02,-0.02); \coordinate (p) at (0.02,-0.20);
                        \draw[treeedge] (s)--(a); \draw[treeedge] (b)--(t); \draw[treeedge] (a)--(b); \draw[treeedge] (c)--(t); \draw[treeedge] (c)--(p);
                        \fill[black] (s) circle (1.20pt); \fill[black] (a) circle (1.20pt); \fill[black] (b) circle (1.20pt); \fill[black] (t) circle (1.20pt); \fill[black] (c) circle (1.20pt); \fill[black] (p) circle (1.20pt);
                    \end{scope}

                    \begin{scope}[shift={(n12)},scale=0.86]
                        \node[circle,draw=black,fill=white,line width=0.25pt,minimum size=7.50mm,inner sep=0pt] at (0,0) {};
                        \coordinate (s) at (-0.27,0.02); \coordinate (a) at (-0.10,0.18); \coordinate (b) at (0.12,0.18); \coordinate (t) at (0.29,0.02); \coordinate (c) at (0.02,-0.02); \coordinate (p) at (0.02,-0.20);
                        \draw[treeedge] (s)--(a); \draw[treeedge] (b)--(t); \draw[treeedge] (a)--(c); \draw[treeedge] (c)--(t); \draw[treeedge] (c)--(p);
                        \fill[black] (s) circle (1.20pt); \fill[black] (a) circle (1.20pt); \fill[black] (b) circle (1.20pt); \fill[black] (t) circle (1.20pt); \fill[black] (c) circle (1.20pt); \fill[black] (p) circle (1.20pt);
                    \end{scope}

                    \begin{scope}[shift={(n13)},scale=0.86]
                        \node[circle,draw=black,fill=white,line width=0.25pt,minimum size=7.50mm,inner sep=0pt] at (0,0) {};
                        \coordinate (s) at (-0.27,0.02); \coordinate (a) at (-0.10,0.18); \coordinate (b) at (0.12,0.18); \coordinate (t) at (0.29,0.02); \coordinate (c) at (0.02,-0.02); \coordinate (p) at (0.02,-0.20);
                        \draw[treeedge] (s)--(a); \draw[treeedge] (a)--(b); \draw[treeedge] (a)--(c); \draw[treeedge] (c)--(t); \draw[treeedge] (c)--(p);
                        \fill[black] (s) circle (1.20pt); \fill[black] (a) circle (1.20pt); \fill[black] (b) circle (1.20pt); \fill[black] (t) circle (1.20pt); \fill[black] (c) circle (1.20pt); \fill[black] (p) circle (1.20pt);
                    \end{scope}

                    \begin{scope}[shift={(n14)},scale=0.86]
                        \node[circle,draw=black,fill=white,line width=0.25pt,minimum size=7.50mm,inner sep=0pt] at (0,0) {};
                        \coordinate (s) at (-0.27,0.02); \coordinate (a) at (-0.10,0.18); \coordinate (b) at (0.12,0.18); \coordinate (t) at (0.29,0.02); \coordinate (c) at (0.02,-0.02); \coordinate (p) at (0.02,-0.20);
                        \draw[treeedge] (a)--(t); \draw[treeedge] (s)--(b); \draw[treeedge] (b)--(t); \draw[treeedge] (a)--(c); \draw[treeedge] (c)--(p);
                        \fill[black] (s) circle (1.20pt); \fill[black] (a) circle (1.20pt); \fill[black] (b) circle (1.20pt); \fill[black] (t) circle (1.20pt); \fill[black] (c) circle (1.20pt); \fill[black] (p) circle (1.20pt);
                    \end{scope}

                    \begin{scope}[shift={(n15)},scale=0.86]
                        \node[circle,draw=black,fill=white,line width=0.25pt,minimum size=7.50mm,inner sep=0pt] at (0,0) {};
                        \coordinate (s) at (-0.27,0.02); \coordinate (a) at (-0.10,0.18); \coordinate (b) at (0.12,0.18); \coordinate (t) at (0.29,0.02); \coordinate (c) at (0.02,-0.02); \coordinate (p) at (0.02,-0.20);
                        \draw[treeedge] (a)--(t); \draw[treeedge] (s)--(b); \draw[treeedge] (b)--(t); \draw[treeedge] (c)--(t); \draw[treeedge] (c)--(p);
                        \fill[black] (s) circle (1.20pt); \fill[black] (a) circle (1.20pt); \fill[black] (b) circle (1.20pt); \fill[black] (t) circle (1.20pt); \fill[black] (c) circle (1.20pt); \fill[black] (p) circle (1.20pt);
                    \end{scope}

                    \begin{scope}[shift={(n16)},scale=0.86]
                        \node[circle,draw=black,fill=white,line width=0.25pt,minimum size=7.50mm,inner sep=0pt] at (0,0) {};
                        \coordinate (s) at (-0.27,0.02); \coordinate (a) at (-0.10,0.18); \coordinate (b) at (0.12,0.18); \coordinate (t) at (0.29,0.02); \coordinate (c) at (0.02,-0.02); \coordinate (p) at (0.02,-0.20);
                        \draw[treeedge] (a)--(t); \draw[treeedge] (s)--(b); \draw[treeedge] (a)--(b); \draw[treeedge] (a)--(c); \draw[treeedge] (c)--(p);
                        \fill[black] (s) circle (1.20pt); \fill[black] (a) circle (1.20pt); \fill[black] (b) circle (1.20pt); \fill[black] (t) circle (1.20pt); \fill[black] (c) circle (1.20pt); \fill[black] (p) circle (1.20pt);
                    \end{scope}

                    \begin{scope}[shift={(n17)},scale=0.86]
                        \node[circle,draw=black,fill=white,line width=0.25pt,minimum size=7.50mm,inner sep=0pt] at (0,0) {};
                        \coordinate (s) at (-0.27,0.02); \coordinate (a) at (-0.10,0.18); \coordinate (b) at (0.12,0.18); \coordinate (t) at (0.29,0.02); \coordinate (c) at (0.02,-0.02); \coordinate (p) at (0.02,-0.20);
                        \draw[treeedge] (a)--(t); \draw[treeedge] (s)--(b); \draw[treeedge] (a)--(b); \draw[treeedge] (c)--(t); \draw[treeedge] (c)--(p);
                        \fill[black] (s) circle (1.20pt); \fill[black] (a) circle (1.20pt); \fill[black] (b) circle (1.20pt); \fill[black] (t) circle (1.20pt); \fill[black] (c) circle (1.20pt); \fill[black] (p) circle (1.20pt);
                    \end{scope}

                    \begin{scope}[shift={(n18)},scale=0.86]
                        \node[circle,draw=black,fill=white,line width=0.25pt,minimum size=7.50mm,inner sep=0pt] at (0,0) {};
                        \coordinate (s) at (-0.27,0.02); \coordinate (a) at (-0.10,0.18); \coordinate (b) at (0.12,0.18); \coordinate (t) at (0.29,0.02); \coordinate (c) at (0.02,-0.02); \coordinate (p) at (0.02,-0.20);
                        \draw[treeedge] (s)--(b); \draw[treeedge] (b)--(t); \draw[treeedge] (a)--(b); \draw[treeedge] (a)--(c); \draw[treeedge] (c)--(p);
                        \fill[black] (s) circle (1.20pt); \fill[black] (a) circle (1.20pt); \fill[black] (b) circle (1.20pt); \fill[black] (t) circle (1.20pt); \fill[black] (c) circle (1.20pt); \fill[black] (p) circle (1.20pt);
                    \end{scope}

                    \begin{scope}[shift={(n19)},scale=0.86]
                        \node[circle,draw=black,fill=white,line width=0.25pt,minimum size=7.50mm,inner sep=0pt] at (0,0) {};
                        \coordinate (s) at (-0.27,0.02); \coordinate (a) at (-0.10,0.18); \coordinate (b) at (0.12,0.18); \coordinate (t) at (0.29,0.02); \coordinate (c) at (0.02,-0.02); \coordinate (p) at (0.02,-0.20);
                        \draw[treeedge] (s)--(b); \draw[treeedge] (b)--(t); \draw[treeedge] (a)--(b); \draw[treeedge] (c)--(t); \draw[treeedge] (c)--(p);
                        \fill[black] (s) circle (1.20pt); \fill[black] (a) circle (1.20pt); \fill[black] (b) circle (1.20pt); \fill[black] (t) circle (1.20pt); \fill[black] (c) circle (1.20pt); \fill[black] (p) circle (1.20pt);
                    \end{scope}

                    \begin{scope}[shift={(n20)},scale=0.86]
                        \node[circle,draw=black,fill=white,line width=0.25pt,minimum size=7.50mm,inner sep=0pt] at (0,0) {};
                        \coordinate (s) at (-0.27,0.02); \coordinate (a) at (-0.10,0.18); \coordinate (b) at (0.12,0.18); \coordinate (t) at (0.29,0.02); \coordinate (c) at (0.02,-0.02); \coordinate (p) at (0.02,-0.20);
                        \draw[treeedge] (s)--(b); \draw[treeedge] (b)--(t); \draw[treeedge] (a)--(c); \draw[treeedge] (c)--(t); \draw[treeedge] (c)--(p);
                        \fill[black] (s) circle (1.20pt); \fill[black] (a) circle (1.20pt); \fill[black] (b) circle (1.20pt); \fill[black] (t) circle (1.20pt); \fill[black] (c) circle (1.20pt); \fill[black] (p) circle (1.20pt);
                    \end{scope}

                    \begin{scope}[shift={(n21)},scale=0.86]
                        \node[circle,draw=black,fill=white,line width=0.25pt,minimum size=7.50mm,inner sep=0pt] at (0,0) {};
                        \coordinate (s) at (-0.27,0.02); \coordinate (a) at (-0.10,0.18); \coordinate (b) at (0.12,0.18); \coordinate (t) at (0.29,0.02); \coordinate (c) at (0.02,-0.02); \coordinate (p) at (0.02,-0.20);
                        \draw[treeedge] (s)--(b); \draw[treeedge] (a)--(b); \draw[treeedge] (a)--(c); \draw[treeedge] (c)--(t); \draw[treeedge] (c)--(p);
                        \fill[black] (s) circle (1.20pt); \fill[black] (a) circle (1.20pt); \fill[black] (b) circle (1.20pt); \fill[black] (t) circle (1.20pt); \fill[black] (c) circle (1.20pt); \fill[black] (p) circle (1.20pt);
                    \end{scope}
                \end{tikzpicture}%
            }
            \caption{Spanning tree state graph.}
        \end{subfigure}
        \hfill
        \begin{subfigure}[t]{0.25\textwidth}
            \centering
            \resizebox{\linewidth}{!}{%
                \begin{tikzpicture}[
                    scale=0.64,
                    x={(1.12cm,0.00cm)},
                    y={(0.55cm,1.00cm)},
                    z={(0.00cm,0.88cm)},
                    line cap=round,line join=round,
                    treeedge/.style={black,line width=0.60pt},
                    face/.style={draw=black!45,fill=blue!22,fill opacity=0.28,line width=0.25pt}
                ]
                    \coordinate (Pone)   at ( 0.640, 0.127,-0.976);
                    \coordinate (Ptwo)   at ( 0.021, 0.718, 0.806);
                    \coordinate (Pthree) at ( 1.267, 0.094, 0.490);
                    \coordinate (Pfour)  at (-0.795, 0.887,-0.534);
                    \coordinate (Pfive)  at ( 0.151,-1.294,-0.114);
                    \coordinate (Psix)   at (-1.284,-0.533, 0.328);

                    \filldraw[face] (Pfive)--(Pthree)--(Psix)--cycle;
                    \filldraw[face] (Pone)--(Pfour)--(Pthree)--cycle;
                    \filldraw[face] (Pone)--(Pfive)--(Pthree)--cycle;
                    \filldraw[face] (Ptwo)--(Pfour)--(Psix)--cycle;
                    \filldraw[face] (Ptwo)--(Pthree)--(Psix)--cycle;
                    \filldraw[face] (Ptwo)--(Pfour)--(Pthree)--cycle;
                    \filldraw[face] (Pone)--(Pfive)--(Psix)--cycle;
                    \filldraw[face] (Pone)--(Pfour)--(Psix)--cycle;

                    \draw[black,line width=0.48pt]
                        (Pone)--(Pthree) (Pone)--(Pfour) (Pone)--(Pfive) (Pone)--(Psix)
                        (Ptwo)--(Pthree) (Ptwo)--(Pfour) (Ptwo)--(Psix)
                        (Pthree)--(Pfour) (Pthree)--(Pfive) (Pthree)--(Psix)
                        (Pfour)--(Psix)
                        (Pfive)--(Psix);

                    \begin{scope}[shift={(Pone)},x=1cm,y=1cm]
                        \fill[black] (0,0) circle (2.8pt);
                    \end{scope}

                    \begin{scope}[shift={(Ptwo)},x=1cm,y=1cm]
                        \fill[black] (0,0) circle (2.8pt);
                    \end{scope}

                    \begin{scope}[shift={(Pthree)},x=1cm,y=1cm]
                        \fill[black] (0,0) circle (2.8pt);
                    \end{scope}

                    \begin{scope}[shift={(Pfour)},x=1cm,y=1cm]
                        \fill[black] (0,0) circle (2.8pt);
                    \end{scope}

                    \begin{scope}[shift={(Pfive)},x=1cm,y=1cm]
                        \fill[black] (0,0) circle (2.8pt);
                    \end{scope}

                    \begin{scope}[shift={(Psix)},x=1cm,y=1cm]
                        \fill[black] (0,0) circle (2.8pt);
                    \end{scope}
                \end{tikzpicture}%
            }
            \caption{Illustration of the $(\mu,\nu)$-polytope.}
        \end{subfigure}

        \caption{A six-vertex example showing the base graph, the spanning tree state graph, and an illustration of the corresponding tree polytope, as defined in~\cref{def:flow-polytope}.}
        \label{fig:flow-polytope}
    \end{figure}

    The Glauber dynamics on spanning trees (and more generally basis exchange walks on matroids) have been studied extensively in the literature in connection with the design and analysis of sampling and counting algorithms. Based on the expansion theory of Feder and Mihail for balanced matroids~\cite{feder1992balanced}, and continuing with the seminal work of Anari--Liu--Oveis Gharan--Vinzant~\cite{anari2024log}, Alev--Lau~\cite{alev2020improved}, and Cryan--Guo--Mousa~\cite{cryan2021modified}, the spectral and mixing time properties of the process $(T_t)_{t\ge 0}$ are now relatively well understood. It is known that the Glauber dynamics on spanning trees mixes in polynomial time in the size of the graph $G$.

    
    The main goal of this paper is to make precise the connections between the Glauber dynamics and optimal transport on finite graphs. We introduce a so-called \emph{$(\mu,\nu)$-tree polytope}, denoted by $\flowpoly(\mu,\nu)$. This is constructed by associating to each tree $T\in\setoftrees$ a canonical flow vector $\embedding[T]$ (see~\cref{def:canonical-embedding}), and then taking the convex hull of all such embedded flows. This polytope lets us replace questions about a discrete optimization landscape on $\stategraph$ with questions about faces, vertices, and edges of a convex polytope. Namely, we prove that the vertices of the polytope are exactly the tree-supported flows. We then describe the fibers of the embedding, and show that edges in the polytope lift to links in the state graph. 

    Below we summarize some of our main results. 

    \begin{proposition}\label{prop:intro-vertices-of-polytope}
        Let $G=(V,E)$ be a connected undirected graph and let $\mu,\nu:V\to\mathbb{R}$ be fixed probability density vectors. For each spanning tree $T\in\setoftrees$, let $\embedding[T]$ denote the canonical embedding vector of $T$ as in~\cref{def:canonical-embedding}. Then the following statements hold.
            \begin{enumerate}[label=(\roman*)]
                \item For each spanning tree $T\in\setoftrees$, $\embedding[T]$ is a vertex of $\flowpoly(\mu,\nu)$.
                \item For each vertex $K\in\flowpoly(\mu,\nu)$, there exists at least one spanning tree $T\in\setoftrees$ such that $K = \embedding[T]$. Moreover, the support of $K$, i.e., edges $\{x,y\}\in E$ such that $K(x,y)>0$, forms a forest $F$ in $G$, and the fiber $\embedding^{-1}[K]$ consists of all spanning trees $T$ of $G$ containing $F$.
            \end{enumerate}
    \end{proposition}

    \Cref{prop:intro-vertices-of-polytope} is proved in~\Cref{sec:flow-polytope}.

    \begin{proposition}\label{prop:intro-connected-fibers}
        Let $G=(V,E)$ be a connected graph and let $F\subseteq E$ be a forest in $G$. Let
            \begin{align*}
                \setoftrees[F] &= \{T\in\setoftrees : T\supseteq F\}.
            \end{align*}
        Then the induced subgraph $\stategraph[\setoftrees[F]]$ of the spanning tree state graph $\stategraph$ is connected.
    \end{proposition}

    \Cref{prop:intro-connected-fibers} is proved in~\Cref{sec:connectivity-minimizers}.

    \begin{theorem}\label{thm:intro-connectivity}
        Let $G=(V,E)$ be a connected graph. For fixed probability density vectors $\mu,\nu:V\to\mathbb{R}$, let
            \begin{align*}
                X := \{T\in\setoftrees : \wass_T(\mu,\nu)=\wass_G(\mu,\nu)\}.
            \end{align*}
        Then the induced subgraph $\stategraph[X]$ is connected.
    \end{theorem}

    \Cref{thm:intro-connectivity} is proved in~\cref{sec:connectivity-minimizers}. Our next result shows that local minimizers, defined below as so-called basins in the state graph, must also be global minimizers. 

    \begin{definition}\label{def:basin}
        Let $G=(V,E)$ be a connected graph and let $g\in\mathbb{R}^{\setoftrees}$ be fixed. Let $c\in\mathbb{R}$. A nonempty subset of spanning trees $X\subseteq \setoftrees$ is called a \emph{basin of $g$ on $\stategraph$ at level $c$} if the following hold:
            \begin{enumerate}[label=(\roman*)]
                \item The induced subgraph $\stategraph[X]$ is connected,
                \item $g(T) = c$ for each $T\in X$,
                \item $g(S) > c$ for each $S\in\partial X$,
            \end{enumerate}
        where $\partial{X}$ is the vertex boundary of $X$ in $\stategraph$, defined by
            \begin{align*}
                \partial X &= \{S\in\setoftrees\setminus X \;:\; \exists\, T\in X \text{ such that } S\sim T\}.
            \end{align*}
    \end{definition}

    \begin{theorem}\label{thm:intro-basins}
        Let $G=(V,E)$ be connected. For fixed probability density vectors $\mu,\nu:V\to\mathbb{R}$, if $X\subseteq\setoftrees$ is a basin of $T\mapsto\wass_T(\mu,\nu)$ in $\stategraph$ at level $c$, then $c=\wass(\mu,\nu)$.
    \end{theorem}

    \Cref{thm:intro-basins} is proved in~\cref{sec:no-local-minima}, leading to a steepest descent algorithm for solving the optimal transport problem on $G$. Starting from an arbitrary initial tree $T_0\in\setoftrees$, we repeatedly move to a neighboring tree with strictly smaller transport cost until we arrive at a tree with cost $\wass_G(\mu,\nu)$. We show that if the input measures are assumed to have discrete values, then this procedure converges to a global minimizer in a number of steps which is polynomial in the size of the vertex set of $G$.

    \begin{theorem}\label{thm:steepest-descent-empirical}
        Let $G=(V,E)$ be a connected graph and let $\mu,\nu:V\to\mathbb{R}$ be fixed probability density vectors. Assume $\mu,\nu$ are \emph{discrete}, i.e., for some $\delta>0$, $\mu-\nu\in\delta\mathbb{Z}^{V}$. Then there exists a path of steepest descent for the transport function in $\stategraph$ arriving at the optimal value $\wass(\mu,\nu)$ in at most $\frac{n-1}{\delta}$ steps.
    \end{theorem}

    \Cref{thm:steepest-descent-empirical} is proved in~\cref{sec:steepest-descent}.

    As an application of the preceding theorems, in~\cref{sec:ollivier-ricci-curvature}, we consider the problem of computing the Ollivier--Ricci curvature of a graph. By combining our steepest descent framework with appropriate choices for the initial spanning trees, we obtain polynomial-time bounds for computing the curvature for each edge in the graph.

    The paper is organized as follows. In \cref{sec:flow-polytope} we define the tree polytope and the canonical embedding. In \cref{sec:connectivity-minimizers} we prove connectivity of global minimizers in the state graph. In \cref{sec:no-local-minima} we introduce basins and rule out suboptimal local minima. In \cref{sec:steepest-descent} we analyze steepest descent and prove a polynomial step bound. In~\cref{sec:ollivier-ricci-curvature} we apply our results to the problem of computing Ollivier--Ricci curvature on graphs. 

    \section{The tree polytope}\label{sec:flow-polytope}

    For $G=(V,E)$ fixed, we consider the associated bi-oriented graph with directed edges given by
        \begin{align}\label{eq:bi-oriented-edges}
            E' &= \{(x,y), (y,x) \,:\, \{x,y\}\in E\}.
        \end{align}
    We make use of the bi-oriented incidence matrix $B\in\mathbb{R}^{V\times E'}$, with entries given by
        \begin{align*}
            B_{v, e} &:= \begin{cases}
                +1 & \text{if } e = (v,w) \text{ for some }w\in V,\\
                -1 & \text{if } e = (w,v) \text{ for some }w\in V,\\
                0 & \text{otherwise}.
            \end{cases}
        \end{align*}
    For probability density functions $\mu,\nu:V\to\mathbb{R}$, a function $J:E'\to\mathbb{R}$ is called a $(\mu,\nu)$-feasible flow, or simply a \emph{flow}, if $J(e)\ge 0$ for each $e\in E'$ and $BJ = \mu-\nu$. The following lemma states that the optimal transport cost $\wass_G(\mu,\nu)$ can be expressed as the minimum of a linear functional over the set of all $(\mu,\nu)$-feasible flows.

    \begin{lemma}\label{lem:flow-formulation}
        Let $G=(V,E)$ be a connected graph, and let $\mu,\nu\in\prob(V)$ be fixed. Then the following holds:
            \begin{align}\label{eq:flow-formulation}
                \wass_G(\mu,\nu) &= \min\left\{\sum_{e\in E'} J(e) \,:\, J\in\mathbb{R}^{E'},\, BJ = \mu-\nu,\, J(e)\ge 0\text{ for each }e\in E'\right\}.
            \end{align}
    \end{lemma}

    \cref{lem:flow-formulation} can be proved, for example, by applying strong duality for linear programs, and we omit the proof here (see, e.g.,~\cite[Ch. 6]{peyre2019computational}, see also~\cite{essid2018quadratically}). Furthermore, in the case of a tree graph, the optimal flow achieving the value of the program~\cref{eq:flow-formulation} is uniquely determined. 

    \begin{lemma}\label{lem:transport-on-trees-formula}
        Let $T=(V,E)$ be a tree and let $\mu,\nu\in\prob(V)$ be fixed. Let $E'$ denote the bi-oriented edges of $T$ as in~\cref{eq:bi-oriented-edges}. Then the minimizer of the program~\cref{eq:flow-formulation} is unique and for each $(x,y)\in E'$ it is given by
            \begin{align}\label{eq:tree-flow-formula}
                J_T(x,y) &= \max\left\{\sum_{z\in S_x}(\mu(z)-\nu(z)),0\right\},
            \end{align}
        where $S_x$ denotes the connected component of $E\setminus\{\{x,y\}\}$ containing $x$.
    \end{lemma}

    For a proof of~\cref{lem:transport-on-trees-formula} we refer to, e.g.,~\cite{bapat1997moore} (see also~\cite{robertson2024all}). By using~\cref{eq:tree-flow-formula} specifically, we may associate to each spanning tree of $G$ a distinguished feasible flow as follows.

    \begin{definition}[Canonical embedded tree flow]\label{def:canonical-embedding}
        Let $G=(V,E)$ be connected, let $\mu,\nu\in\prob(V)$ be fixed probability density functions, and let $T\in\setoftrees$ be a fixed spanning tree of $G$. Let $T'$ denote the bi-oriented edges of $T$ as in~\cref{eq:bi-oriented-edges}. For $(x,y)\in T'$, let $S_x$ denote the connected component of $T$ containing $x$ after removing the edge $\{x,y\}$. Define the \emph{canonical embedding} of $T$ to be the vector $\embedding[T]\in \mathbb{R}_{\ge 0}^{E'}$ as follows:
            \begin{align}\label{eq:canonical-embedding}
                \embedding[T](x,y) &= \begin{cases}
                    \max\left\{\sum_{z\in S_x} (\mu(z)-\nu(z)), 0\right\}, & \text{if } \{x,y\}\in T,\\
                    0, & \text{if } \{x,y\}\notin T.
                \end{cases}
            \end{align}
    \end{definition}

    \begin{remark}\label{rmk:canonical-embedding}
        The embedding vector $\embedding[T]$ may equivalently be defined as follows. Fix a distinguished orientation of $T$ and then let $\widetilde{B}$ denote the restriction of the bi-oriented incidence matrix $B$ to the columns indexed by the oriented edges of $T$. Since $\widetilde{B}$ has trivial kernel, the linear system $\widetilde{B} X = \mu-\nu$ admits a unique solution $X$. 
        Then $\embedding[T]$ as in~\cref{eq:canonical-embedding} is recovered by setting $\embedding[T](x,y)=\max\{X(x,y),0\}$ and $\embedding[T](y,x)=\max\{-X(x,y),0\}$ for each oriented edge $(x,y)$ of $T$, and by assigning $0$ to all remaining coordinates in $E'$. 
    \end{remark}

    \begin{definition}\label{def:flow-polytope}
        Let $G=(V,E)$ be a connected graph, and let $\mu,\nu\in\prob(V)$ be fixed. The \emph{$(\mu,\nu)$-tree polytope,} or simply the tree polytope, is given by
            \begin{align*}
                \flowpoly(\mu,\nu) &:= \mathrm{conv}\left(\{\embedding[T] \,:\, T\in\mathcal{T}(G)\} \right),
            \end{align*}
        where $\mathrm{conv}(\cdot)$ denotes the convex hull of a set in Euclidean space.
    \end{definition}    

    The advantage of using the tree polytope is that the problem of determining the transportation distance $\wass_G(\mu,\nu)$ can now be modeled on the vertices of $\flowpoly(\mu,\nu)$. As it turns out, the skeleton structure of $\flowpoly(\mu,\nu)$ is intimately related to that of the spanning tree state graph $\stategraph$. We state and prove the following proposition, which makes this precise.

    \begin{proposition}\label{prop:vertices-of-polytope}
        Let $G=(V,E)$ be a connected graph and let $\mu,\nu\in\prob(V)$ be fixed. Then the following statements hold.
            \begin{enumerate}[label=(\roman*)]
                \item For each spanning tree $T\in\setoftrees$, $\embedding[T]$ is a vertex of $\flowpoly(\mu,\nu)$.
                \item For each vertex $K\in\flowpoly(\mu,\nu)$, there exists at least one spanning tree $T\in\setoftrees$ such that $K = \embedding[T]$. Moreover, the support of $K$, i.e., edges $\{x,y\}\in E$ such that $K(x,y)>0$, forms a forest $F$ in $G$, and the fiber $\embedding^{-1}[K]$ consists of all spanning trees $T$ of $G$ containing $F$.
            \end{enumerate}
    \end{proposition}

    \begin{proof}
        For claim \emph{(i)}, fix a spanning tree $T\in\setoftrees$. Suppose $K=\embedding[T]$ is not a vertex in $\flowpoly(\mu,\nu)$. Then there exist $K_1,K_2\in\flowpoly(\mu,\nu)$ with $K_1\neq K_2$ and
            \begin{align*}
                K = \frac{1}{2}(K_1+K_2).
            \end{align*}
        In particular, $K_1,K_2\ge 0$, and thus $K(e)=0$ implies $K_1(e)=K_2(e)=0$. Let
            \begin{align*}
                \overrightarrow{T} := \{(x,y)\in E' : K(x,y) > 0\}.
            \end{align*}
        By~\cref{eq:canonical-embedding}, $\overrightarrow{T}$ is an oriented forest contained in $T$, and each of $K,K_1,K_2$ are feasible flows with supports contained in $\overrightarrow{T}$. 

        Consider the vector subspace $A\subseteq\mathbb{R}^{E'}$ given by
            \begin{align*}
                A &:= \mathrm{span} \left\{ \mathbf{1}_{(x,y)} \,:\, (x,y)\in\overrightarrow{T} \right\},
            \end{align*}
        where $\mathbf{1}_{(x,y)}$ is the standard basis vector in $\mathbb{R}^{E'}$ corresponding to the directed edge $(x,y)$. We claim that the restriction of $B$ to $A$, i.e., the linear map $B|_{A}:A\to\mathbb{R}^{V}$ given by $J\mapsto BJ$, is injective. 

        We argue by induction on $m:=|\overrightarrow{T}|$. For the base case $m=0$, we have $A=\{0\}$, so $B|_A$ is injective. For the induction step, assume $m\ge 1$ and that the claim holds for every oriented forest with fewer than $m$ edges. Let $J\in A$ satisfy $BJ=0$. Since the underlying graph of $\overrightarrow{T}$ is a nonempty forest, it contains a leaf $v\in V$ incident to a unique oriented edge $e_0\in\overrightarrow{T}$. Because $J$ vanishes outside $\overrightarrow{T}$, the $v$-th coordinate of $BJ$ is $\pm J(e_0)$, where the sign depends on the orientation of $e_0$. Hence $J(e_0)=0$. Let
            \begin{align*}
                \overrightarrow{T}_1 := \overrightarrow{T}\setminus\{e_0\}
            \end{align*}
        and set
            \begin{align*}
                A_1 := \mathrm{span}\left\{\mathbf{1}_{(x,y)}:\,(x,y)\in\overrightarrow{T}_1\right\}.
            \end{align*}
        Then $J\in A_1$, and $\overrightarrow{T}_1$ is again an oriented forest with $m-1$ edges. By the induction hypothesis, the restriction of $B$ to $A_1$ is injective, so $BJ=0$ implies $J=0$. This proves that $B|_A$ is injective.

        Since $K,K_1,K_2\in A$ and each is a feasible flow, we have
            \begin{align*}
                BK = BK_1 = BK_2 = \mu-\nu.
            \end{align*}
        Injectivity of $B|_A$ therefore gives $K=K_1=K_2$, contradicting $K_1\neq K_2$. This shows $\embedding[T]$ is a vertex of the polytope $\flowpoly(\mu,\nu)$, proving claim \emph{(i)}.

        For claim \emph{(ii)}, since $\flowpoly(\mu,\nu)$ is the convex hull of $\{\embedding[T]:T\in\setoftrees\}$, a vertex $K$ of $\flowpoly(\mu,\nu)$ is necessarily one of these points and hence $K=\embedding[T]$ for some spanning tree $T\in\setoftrees$. Let $F = \{\{x,y\}\in E : K(x,y) > 0\}$. By~\cref{eq:canonical-embedding}, the set $F$ is contained in $T$, hence $F$ is a forest.

        We next show that
            \begin{align*}
                \embedding^{-1}[K] = \{S\in\setoftrees : S\supseteq F\}.
            \end{align*}
        First, if $\embedding[S]=K$, then every positive coordinate of $K$ is supported on an edge of $S$, so necessarily $S\supseteq F$.

        Conversely, let $S\in\setoftrees$ satisfy $S\supseteq F$. Orient the edges of $S$ so that each edge of $F$ points in the direction where $K$ is positive, and orient the remaining edges arbitrarily. Now assign each edge of $F$ the corresponding value of $K$, and assign each edge of $S\setminus F$ the value $0$. This is just the flow $K$ written in the chosen orientation, so it is feasible and its positive part is exactly $K$. By the uniqueness statement in~\cref{rmk:canonical-embedding}, this is the signed flow associated with $S$, and therefore $\embedding[S]=K$. This proves the reverse direction and completes the proof of claim \emph{(ii)}.
    \end{proof}

\section{Properties of the tree polytope}\label{sec:connectivity-minimizers}

    In this section, we use the tree polytope to prove that the global minimizers of the function $T\mapsto \wass_T(\mu,\nu)$ form a connected set in the state graph $\stategraph$. To do so, we state and prove a lemma about the connectivity structure of the state graph.

    \begin{proposition}\label{prop:connected-fibers}
        Let $G=(V,E)$ be a connected graph and let $F\subseteq E$ be a forest in $G$. Let
            \begin{align*}
                \setoftrees[F] &= \{T\in\setoftrees : T\supseteq F\}.
            \end{align*}
        Then the induced subgraph $\stategraph[\setoftrees[F]]$ of the spanning tree state graph $\stategraph$ is connected.
    \end{proposition}

    \begin{proof}
        If $|F|=|V|-1$, then $F$ is already a spanning tree and $\setoftrees[F]=\{F\}$, so the claim is trivial. Assume $|F|=k<|V|-1$. We will show that for any $T,T'\in\setoftrees[F]$, there is a path in $\stategraph[\setoftrees[F]]$ connecting $T$ and $T'$.

        If $T=T'$, there is nothing to prove. Otherwise choose an edge $e\in T\setminus T'$. Set $F_1 := T\setminus\{e\}$. Then $F_1$ is a forest with two connected components on vertex sets $A_1$ and $B_1$. Since $T'$ is connected, it must contain at least one edge in the cut $(A_1,B_1)$. Because $e$ is the unique edge of $T$ crossing this cut, any such edge lies in $T'\setminus T$. Choose one such edge $f\in E(A_1,B_1)\cap T'$. Then
            \begin{align*}
                T_1 := T\setminus\{e\}\cup\{f\}
            \end{align*}
        is a spanning tree, and we note $T_1$ still contains $F$ because $e\notin F$. 
        Moreover, since $e\notin T'$ and $f\in T'$, we have $T_1\cap T' = (T\cap T')\cup\{f\}$, and therefore $|T_1\cap T'| = |T\cap T'|+1$. Repeating this process increases the number of edges shared with $T'$ by one at each step. Since a spanning tree has only $|V|-1$ edges, after finitely many steps we must arrive at $T'$, producing a path from $T$ to $T'$ in $\stategraph[\setoftrees[F]]$.
    \end{proof}

    Next, we show that edges in $\flowpoly(\mu,\nu)$ are lifted in a straightforward way to links in $\stategraph$, as follows.

    \begin{lemma}\label{lem:adjacency-in-K}
        Let $G=(V,E)$ be a connected graph and let $\mu,\nu\in\prob(V)$ be fixed. Let $K_1, K_2\in\flowpoly(\mu,\nu)$ be vertices such that $\{K_1, K_2\}$ is an edge ($1$-face) of $\flowpoly(\mu,\nu)$. If $T \in\setoftrees$ satisfies $\embedding[T] = K_1$, then there exists $T'\in\setoftrees$ such that $\embedding[T'] = K_2$ and such that $T$ and $T'$ are adjacent in the state graph $\stategraph$.
    \end{lemma}

    \begin{proof}
        Let $K_1, K_2\in\flowpoly(\mu,\nu)$ be vertices such that $\{K_1,K_2\}$ is an edge of $\flowpoly(\mu,\nu)$, and let $T\in\setoftrees$ satisfy $\embedding[T]=K_1$. By~\cref{prop:vertices-of-polytope}(ii), there exist forests $F_1,F_2\subseteq E$, obtained as the supports of $K_1$ and $K_2$, respectively, and with the property that
            \begin{align}\label{eq:fiber-description-adjacency}
                \embedding^{-1}[K_i] = \{S\in\setoftrees: S\supseteq F_i\},\text{ for }\quad i=1,2.
            \end{align}
        In particular, $T\supseteq F_1$. Since $\{K_1,K_2\}$ is an edge, we have $K_1\neq K_2$. Hence the fibers in~\cref{eq:fiber-description-adjacency} are disjoint and in particular there is no spanning tree containing both $F_1$ and $F_2$, so the union $F = F_1\cup F_2$ cannot be a forest. Thus, by passing to a simple subcycle by deleting repeated vertices as needed, $F$ contains at least one simple cycle. We claim that this cycle is unique. Suppose $F$ contains two distinct simple cycles. Let $\overline{K} := \tfrac{1}{2}(K_1+K_2)$ and set
            \begin{align*}
                F^+ := \{(x,y)\in E' : \overline{K}(x,y) > 0\}.
            \end{align*}
        Then $\overline{K}$ is strictly positive on $F^+$ and satisfies $B_{F^+}\overline{K}^{(F^+)} = \mu-\nu$, where $\overline{K}^{(F^+)}$ denotes the restriction of $\overline{K}$ to the coordinates indexed by $F^+$, and where $B_{F^+}$ denotes the restriction of $B$ to the columns indexed by $F^+$. The underlying undirected graph of $F^+$ contains $F$, so if $F$ has two distinct cycles then $\dim\ker(B_{F^+})\ge 2$. Choose linearly independent $h_1,h_2\in\ker(B_{F^+})$. For $|\varepsilon_1|,|\varepsilon_2|$ sufficiently small, the vector
            \begin{align*}
                \widetilde{K} = \overline{K} + \varepsilon_1 h_1 + \varepsilon_2 h_2
            \end{align*}
        remains nonnegative and satisfies $B\widetilde{K} = \mu-\nu$, so $\widetilde{K}\in\flowpoly(\mu,\nu)$. This produces a $2$-dimensional face through $\overline{K}$, contradicting that $\{K_1,K_2\}$ is an edge. Therefore $F$ contains a unique simple cycle, say $C\subseteq F$.
        
        Next we claim that $F_1$ and $F_2$ differ only on $C$; specifically,
            \begin{align*}
                F_1\cap(F\setminus C) = F_2\cap(F\setminus C) = F\setminus C.
            \end{align*}
        To see this, let $g\in F\setminus C$ be an edge off the cycle. We claim that $g\in F_1\cap F_2$. Clearly $g\in F_1$ or $g\in F_2$ holds automatically, so assume without loss of generality that $g\in F_1$. If $g\notin F_2$, then $K_2$ vanishes on both orientations of $g$ while $K_1$ is positive on exactly one orientation $\overrightarrow{g}$ of $g$, so $K_1-K_2$ is nonzero on $\overrightarrow{g}$. Since $B(K_1-K_2)=0$, $F$ contains a unique cycle, so the support of $K_1-K_2$ must contain $C$. This is impossible because $g\notin C$. Therefore $g\in F_2$, and we conclude that $F_1\cap(F\setminus C) = F_2\cap(F\setminus C) = F\setminus C$.

        We claim that $C$ contains at least one edge $e\in F_2$ such that $e\notin T$. Indeed, if no such edge existed, then $C\subseteq F_1\subseteq T$, contradicting that $T$ is a tree. Therefore $C$ must contain an edge avoiding $T$. Since $T$ contains $F_1$ this edge must lie in $F_2$. Fix such an edge $e$. Adding $e$ to $T$ creates a unique simple cycle; denote it by $C_T(e)\subseteq T\cup\{e\}$. Indeed, $C_T(e)$ is exactly the edge $e$ together with the unique simple path in $T$ joining the endpoints of $e$. Every edge of $C_T(e)$ other than $e$ lies in $T$, and $C_T(e)$ must contain at least one edge of $T$ that is not in $F_2$, otherwise $C_T(e)\subseteq (T\cup\{e\})\cap F_2$ would contradict that $F_2$ is a forest. Fix an edge $e'\in C_T(e) \cap (T\setminus F_2)$. Define
            \begin{align*}
                T' \;:=\; (T\cup\{e\})\setminus\{e'\}.
            \end{align*}
        Clearly $T'\in\setoftrees$, and by construction, $T\sim T'$ in $\stategraph$. By~\eqref{eq:fiber-description-adjacency}, it remains to show that $T'\supseteq F_2$. Let $f\in F_2$ be arbitrary. If $f=e$, then $f\in T'$ by construction. Thus we may assume $f\neq e$ and consider two cases.
        
        First, if $f\notin C$, then $f\in F_2\cap(F\setminus C)=F\setminus C$ by construction, hence $f\in F_1$ and therefore $f\in T$ (since $T\supseteq F_1$). Moreover, to produce $T'$ we removed only $e'\in C$, so $f$ remains in $T'$. Thus $f\in T'$.
        
        Second, if $f\in C$, then the only edge of $C$ removed in passing from $T\cup\{e\}$ to $T'$ is $e'$, and by construction $e'\notin F_2$. But $f\in F_2$ so $f\neq e'$ and hence $f\in T'$.

        In both cases we have shown $f\in T'$. Since $f\in F_2$ was arbitrarily chosen, this proves $T'\supseteq F_2$. By~\cref{prop:vertices-of-polytope}(ii), it follows that $\embedding[T']=K_2$. Therefore we have proved that $T\sim T'$ in $\stategraph$, $\embedding[T]=K_1$, and $\embedding[T']=K_2$ as desired.
    \end{proof}

    Finally, using the preceding setup, we will show that the minimizers of the transport function are connected.

    \begin{theorem}\label{thm:connectivity-minimizers}
        Let $G=(V,E)$ be a connected graph and let $\mu,\nu\in\prob(V)$ be fixed. Consider the minimum level set $X = \{T\in\setoftrees : \wass_{T}(\mu,\nu) = \wass(\mu,\nu)\}$ of the transport function $T\mapsto \wass_{T}(\mu,\nu)$. Then the induced subgraph $\stategraph[X]$ of the spanning tree state graph $\stategraph$ is connected.
    \end{theorem}

    \begin{proof}
        Consider the $(\mu,\nu)$-tree polytope $\flowpoly(\mu,\nu)$, and let
            \begin{align*}
                \flowpoly_\star &= \flowpoly(\mu,\nu) \cap \{K\in\mathbb{R}^{E'} : \mathbf{1}^\top K = \wass(\mu,\nu)\}.
            \end{align*}
        We claim that $\flowpoly_\star$ is a face of $\flowpoly(\mu,\nu)$. Recall a basic fact from polytopal geometry: if $F$ is a subset of a convex polytope $P\subseteq\mathbb{R}^{d}$, then $F$ is a face of $P$ if and only if it is convex and satisfies the following condition: if $x\in F$ and $x = \theta y + (1-\theta) z$ with $y, z\in P$ and $0\le \theta\le 1$, then $y, z\in F$. To this end, let $K\in\flowpoly_\star$, $K_1, K_2\in\flowpoly(\mu,\nu)$, and $0 \le \theta \le 1$ be fixed, and assume $K = \theta K_1 + (1-\theta) K_2$. Then
            \begin{align*}
                \mathbf{1}^\top K &= \theta \mathbf{1}^\top K_1 + (1-\theta) \mathbf{1}^\top K_2.
            \end{align*}
        Since $\flowpoly(\mu,\nu)$ is the convex hull of $\{\embedding[T]:T\in\setoftrees\}$, we have $\mathbf{1}^\top K_i \ge \wass(\mu,\nu)$ for $i=1,2$. Together with $\mathbf{1}^\top K = \wass(\mu,\nu)$ and the fact that $K(e)\geq 0$ for each $e\in E'$, it follows that $\mathbf{1}^\top K_i = \wass(\mu,\nu)$ for $i=1,2$, and hence that $K_i\in\flowpoly_\star$. This proves that $\flowpoly_\star$ is a face of $\flowpoly(\mu,\nu)$.

        Now take two trees $S, T\in \setoftrees$ achieving the optimal transport cost between $\mu,\nu$, i.e., $S,T\in X$. We claim that there exists a path in $\stategraph[X]$ connecting them. Clearly $\embedding[S], \embedding[T]\in \flowpoly_\star$. Since $\flowpoly_\star$ is a nonempty face of $\flowpoly(\mu,\nu)$, its $1$-skeleton is connected, so there exists a path
            \begin{align*}
                (K_0 = \embedding[S], K_1,\dots, K_{r-1}, K_{r} = \embedding[T])
            \end{align*}
        of vertices in $\flowpoly_\star$ connecting the flows $\embedding[S], \embedding[T]$. Since $\flowpoly_\star$ is a face of $\flowpoly(\mu,\nu)$, these are also vertices of $\flowpoly(\mu,\nu)$.

        Starting from $T_0 = S$, apply~\cref{lem:adjacency-in-K} to each edge $\{K_{i-1},K_i\}$ to obtain $T_i\in\setoftrees$ with $\embedding[T_i]=K_i$ and $T_{i-1}\sim T_i$ in $\stategraph$. Because $K_i\in\flowpoly_\star$, we have $\mathbf{1}^\top \embedding[T_i] = \wass(\mu,\nu)$, and hence $T_i\in X$. This yields a path
            \begin{align*}
                (T_0 = S, T_1, \dotsc, T_{r-1}, T_{r})
            \end{align*}
        in $\stategraph[X]$ with $\embedding[T_i]=K_i$ for each $0\le i\le r$.

        The endpoint $T_{r}$ need not equal $T$, but they belong to the same fiber of $\embedding$ since $\embedding[T]=\embedding[T_r]$. By~\cref{prop:connected-fibers}, this fiber is connected in $\stategraph$, so we may adjoin a path from $T_{r}$ to $T$. All vertices on this additional path share the same value of $\mathbf{1}^\top \embedding[\cdot]$, hence remain in $X$. The proof is completed.
    \end{proof}

\section{Properties of minimizers of the transport function}\label{sec:no-local-minima}

Before we proceed to show that the function $T\mapsto \wass_T(\mu,\nu)$ on the spanning tree state graph $\stategraph$ has no suboptimal local minima, we first state and prove a useful lemma.

\begin{lemma}[Edge descent at nonoptimal vertices]\label{lem:edge-descent-polytope}
    Let $K_0$ be a vertex of $\flowpoly(\mu,\nu)$ such that $\mathbf{1}^\top K_0 > \wass(\mu,\nu)$. Then there exists a vertex $K_1$ adjacent to $K_0$ in the $1$-skeleton of $\flowpoly(\mu,\nu)$ with
        \begin{align*}
            \mathbf{1}^\top K_1 < \mathbf{1}^\top K_0.
        \end{align*}
\end{lemma}

\begin{proof}
    Let $K_\star\in\flowpoly(\mu,\nu)$ be optimal for the linear objective $K\mapsto \mathbf{1}^\top K$, so that $\mathbf{1}^\top K_\star = \wass(\mu,\nu)$. Set $d:=K_\star-K_0$. Since $\flowpoly(\mu,\nu)$ is convex, we have
        \begin{align*}
            K_0+t d = (1-t)K_0+tK_\star \in \flowpoly(\mu,\nu),\quad t\in[0,1],
        \end{align*}
    so $d$ is a feasible direction at $K_0$.

    For each vertex $K_1',K_2',\dotsc, K_m'$ of $\flowpoly(\mu,\nu)$ adjacent to $K_0$ in its $1$-skeleton, let $d^{(i)}$ denote the edge direction $d^{(i)} = K_{i}'-K_0$, $i=1,2,\dotsc, m$. A standard polytope fact is that the feasible cone at a vertex is generated by its incident edge directions (see, e.g.,~\cite{ziegler1995lectures}), hence there are coefficients $\alpha_i\ge 0$ such that
        \begin{align*}
            d = \sum_{i=1}^m \alpha_i d^{(i)}.
        \end{align*}
    Taking objective inner products gives
        \begin{align*}
            0 > \mathbf{1}^\top d = \mathbf{1}^\top K_\star - \mathbf{1}^\top K_0 = \sum_{i=1}^m \alpha_i\,\mathbf{1}^\top d^{(i)}.
        \end{align*}
    Therefore $\mathbf{1}^\top d^{(i_0)}<0$ for some $i_0$. Let $K_1$ be the other endpoint of the corresponding edge through $K_0$; then $K_1 = K_0 + \tau d^{(i_0)}$ for some $\tau>0$, and
        \begin{align*}
            \mathbf{1}^\top K_1 = \mathbf{1}^\top K_0 + \tau\,\mathbf{1}^\top d^{(i_0)} < \mathbf{1}^\top K_0.
        \end{align*}
\end{proof}

The next result rules out isolated \emph{vertices} of the state graph $\stategraph$ where the transport function might have a strict suboptimal local minimum.

\begin{theorem}\label{thm:isolated-vertex-minimizers}
    Let $G=(V,E)$ be connected, $\mu,\nu\in\prob(V)$, and let $\setoftrees$ denote the set of spanning trees of $G$. Consider the function $f\in\mathbb{R}^{\setoftrees}$ defined by
        \begin{align*}
            f(T) &= \wass_T(\mu,\nu), \quad T\in\setoftrees.
        \end{align*}
    Let $T_0\in\setoftrees$ be fixed. Assume $T_0$ is a strict minimizer of $f$ on its neighborhood in $\stategraph$, i.e.,
        \begin{align}\label{eq:strict-local-min}
            f(S) > f(T_0)\text{ for each }S\sim T_0.
        \end{align}
    Then $T_0$ is a global minimizer of $f$, i.e.,
        \begin{align*}
            f(T_0) = \wass(\mu,\nu).
        \end{align*}
\end{theorem}

 
In fact, a stronger version of~\cref{thm:isolated-vertex-minimizers} holds with vertices replaced by ``basins,'' as defined in~\cref{def:basin}. 

\begin{theorem}\label{thm:no-isolated-basins}
    Let $G=(V,E)$ be a connected graph, $\mu,\nu\in\prob(V)$ be fixed, let $\setoftrees$ denote the set of spanning trees of $G$, and let $\stategraph$ denote the spanning tree state graph. Let $f:\setoftrees\rightarrow\mathbb{R}$ denote the transport function $f(T) = \wass_T(\mu,\nu)$. Let $X\subseteq\setoftrees$ be a basin of $f$ at level $c\in\mathbb{R}$. Then $c = \wass_G(\mu,\nu)$.
\end{theorem}

\begin{proof}[Proof of~\cref{thm:no-isolated-basins}]
    Suppose to the contrary that $c>\wass(\mu,\nu)$. Fix any $T\in X$ and set $K_0:=\embedding[T]$. Then
        \begin{align*}
            \mathbf{1}^\top K_0 = f(T)=c>\wass(\mu,\nu).
        \end{align*}
    By~\cref{lem:edge-descent-polytope}, there exists a vertex $K_1\in\flowpoly(\mu,\nu)$ adjacent to $K_0$ with
        \begin{align*}
            \mathbf{1}^\top K_1 < \mathbf{1}^\top K_0 = c.
        \end{align*}
    Applying~\cref{lem:adjacency-in-K} to the edge $\{K_0,K_1\}$, we obtain $S\in\setoftrees$ such that $S\sim T$ in $\stategraph$ and $\embedding[S]=K_1$. Hence
        \begin{align*}
            f(S)=\mathbf{1}^\top K_1 < c.
        \end{align*}
    But $S\sim T$ with $T\in X$, so either $S\in X$ or $S\in\partial X$. If $S\in X$, then by the definition of basin we must have $f(S)=c$, leading to a contradiction. If $S\in\partial X$, then again by definition we must have $f(S)>c$, similarly leading to a contradiction. Therefore $c\le \wass(\mu,\nu)$.

    Since $\wass(\mu,\nu)$ is the global minimum of $f$, we always have $c\ge \wass(\mu,\nu)$ for any nonempty level set $\{T: f(T)=c\}$. Hence $c=\wass(\mu,\nu)$.
\end{proof}

\section{Convergence of steepest descent on the state graph}\label{sec:steepest-descent}

In this section we analyze a simple steepest descent algorithm for minimizing the transport function $f(T) = \wass_T(\mu,\nu)$ on the spanning tree state graph $\stategraph$. Before defining the algorithm, we state and prove a technical lemma.

\begin{lemma}\label{lem:descent-step}
    Let $G=(V,E)$ be a connected graph and let $\mu,\nu\in\prob(V)$ be fixed. Let $T\in\setoftrees$ be fixed satisfying
        \begin{align*}
            \wass_T(\mu,\nu) > \wass(\mu,\nu).
        \end{align*}
    Then there exists a tree $T'\in\setoftrees$ with $T'\sim T$ in $\stategraph$ such that
        \begin{align*}
            \wass_{T'}(\mu,\nu) < \wass_T(\mu,\nu).
        \end{align*}
\end{lemma}

\begin{proof}
    Set $K_0:=\embedding[T]$. By~\cref{prop:vertices-of-polytope}, $K_0$ is a vertex of $\flowpoly(\mu,\nu)$. By assumption it follows that
        \begin{align*}
            \mathbf{1}^{\top}K_0 > \wass(\mu,\nu).
        \end{align*}
    Applying~\cref{lem:edge-descent-polytope}, there exists a vertex $K_1$ adjacent to $K_0$ in the $1$-skeleton of $\flowpoly(\mu,\nu)$ such that
        \begin{align*}
            \mathbf{1}^{\top}K_1 < \mathbf{1}^{\top}K_0.
        \end{align*}
    Since $K_0=\embedding[T]$ and $\{K_0,K_1\}$ is an edge of $\flowpoly(\mu,\nu)$, \cref{lem:adjacency-in-K} yields a spanning tree $T'\in\setoftrees$ such that $T'\sim T$ and $\embedding[T']=K_1$. By construction, $\mathbf{1}^{\top}\embedding[S]=\wass_S(\mu,\nu)$ for every spanning tree $S$, so it follows that $\wass_{T'}(\mu,\nu)=\mathbf{1}^{\top}K_1<\mathbf{1}^{\top}K_0=\wass_T(\mu,\nu)$. This proves the lemma.
\end{proof}

Using~\cref{lem:descent-step}, the steepest descent algorithm on $\stategraph$ proceeds as follows. Starting from an initial tree $T_0\in\setoftrees$, if no neighbors of $T_0$ have strictly lower cost, then $T_0$ is optimal and we are done. Otherwise there exists a neighbor $T_1\sim T_0$ with $\wass_{T_1}(\mu,\nu) < \wass_{T_0}(\mu,\nu)$. We repeat this process by checking the neighbors of $T_1$. Since the state graph is finite, this process terminates in finitely many steps with an optimal tree. 

\begin{theorem}\label{thm:steepest-descent-convergence}
    Let $G=(V,E)$ be a connected graph and let $\mu,\nu\in\prob(V)$ be fixed with $\mu\ne \nu$. Let $\delta = \delta(\mu,\nu) > 0$ denote the smallest positive improvement gap between two trees with differing transport cost. Then for any initial tree $T\in\setoftrees$ there exists a path $T_0=T,T_1,\dotsc, T_{m}$ in $\stategraph$ satisfying:
        \begin{enumerate}
            \item All steps lead to strict improvement of the transport cost, i.e., $\wass_{T_{i+1}}(\mu,\nu) < \wass_{T_i}(\mu,\nu)$ for each $0\le i \le m-1$;
            \item The length of the path satisfies $m\leq \frac{n-1}{\delta}$;
            \item The path terminates with optimal cost, i.e., $\wass_{T_m}(\mu,\nu) = \wass(\mu,\nu)$.
        \end{enumerate}
\end{theorem}

\begin{proof}
    If $T$ is optimal or if $\delta = +\infty$, the claim holds immediately. We may assume $\wass_{T}(\mu,\nu) > \wass(\mu,\nu)$. 
    
    By~\cref{lem:descent-step}, there exists a neighbor $T_1\sim T$ in $\stategraph$ with $\wass_{T_1}(\mu,\nu) < \wass_{T}(\mu,\nu)$. If $\wass_{T_1}(\mu,\nu)=\wass(\mu,\nu)$, we are done; otherwise at stage $i$, having constructed $T_i$ with $\wass_{T_i}(\mu,\nu) > \wass(\mu,\nu)$, we apply the lemma to obtain $T_{i+1}\sim T_i$ with $\wass_{T_{i+1}}(\mu,\nu) < \wass_{T_i}(\mu,\nu)$.

    This greedy descent process terminates in finitely many steps since each step yields a strict improvement of at least $\delta > 0$. Specifically, starting from $\wass_{T}(\mu,\nu)$, each step decreases the cost by at least $\delta$, so the number of steps is bounded by
        \begin{align*}
            m \leq \frac{\wass_{T}(\mu,\nu) - \wass(\mu,\nu)}{\delta} \leq \frac{\max_{T'\in\setoftrees} \wass_{T'}(\mu,\nu) - \wass(\mu,\nu)}{\delta} \leq \frac{n-1}{\delta},
        \end{align*}
    where the last inequality follows because every spanning tree on $n$ vertices has $n-1$ edges and for each oriented edge $(x,y)$ it follows that the optimal flow $J_{T}$ as in~\cref{eq:tree-flow-formula} satisfies $J_T(x,y)\le 1$. Therefore $\wass_{T'}(\mu,\nu)\le n-1$ for each $T'\in\setoftrees$. By construction, the path $T_0, T_1, \dotsc, T_m$ satisfies all three required properties.
\end{proof}

The minimum improvement gap $\delta$ may be arbitrarily small in general, but the following result gives a nontrivial lower bound in the case where $\mu$ and $\nu$ are discrete measures, which arise in various applications.

\begin{lemma}\label{lem:rational-measures-gap}
    Let $G=(V,E)$ be a connected graph and let $\mu,\nu\in\prob(V)$ be fixed with $\mu\ne \nu$. Assume that for each $x\in V$ it holds $\mu(x),\nu(x) \in \alpha\mathbb{Z}$ for some $\alpha > 0$. Then $\delta(\mu,\nu) \geq \alpha$.
\end{lemma}

\Cref{lem:rational-measures-gap} is a simple consequence of~\cref{lem:transport-on-trees-formula}: for each spanning tree $T$ of $G$, $\wass_T(\mu,\nu)$ is a nonnegative integer multiple of $\alpha$, and thus any two distinct tree transport values differ by at least $\alpha$. 

%
%

\Cref{thm:steepest-descent-empirical} now follows immediately from~\cref{lem:rational-measures-gap} and~\cref{thm:steepest-descent-convergence}.

\section{Ollivier--Ricci curvature}\label{sec:ollivier-ricci-curvature}

Ollivier--Ricci curvature is a transportation-based notion of curvature that quantitatively captures the local geometry of a graph. Roughly speaking, on a graph $G=(V,E)$, if the one-step neighborhoods of two vertices are closer to each other than the vertices themselves, the curvature is positive, while if they are farther apart, the curvature is negative~\cite{ollivier2009ricci}. This quantity has proved useful in a range of applications in graph and network analysis~\cite{tian2025curvature}. Of interest is the problem of computing Ollivier--Ricci curvature in graphs at scale.

In the present setting, our steepest descent framework gives an efficient algorithmic approach to computing the curvature of a graph: the computation at each edge is reduced to a single transportation problem between two explicit probability measures, and the results of this section show that, with an appropriate initial spanning tree, the steepest descent procedure can be used to compute the curvature of a graph in polynomial time.

    We begin by recalling the definition of Ollivier--Ricci curvature in a graph. Write $N(x):=\{z\in V : z\sim x\}$ for the neighbor set of a vertex $x$. For each $x\in V$, define the probability measure $m_x\in\prob(V)$ by
        \begin{align*}
            m_x(z) &= \begin{cases}
                \frac{1}{\deg(x)}, & z\in N(x),\\
                0, & \text{otherwise}.
            \end{cases}
        \end{align*}
    In other words, $m_x$ is the one-step distribution of the simple random walk started at $x$. For distinct vertices $x,y\in V$, the Ollivier--Ricci curvature is defined by
        \begin{align*}
            \kappa(x,y) &= 1-\frac{\wass(m_x,m_y)}{d_G(x,y)}.
        \end{align*}
    In particular,
        \begin{align*}
            \kappa(x,y) &= 1-\wass(m_x,m_y),\text{ for each }\{x,y\}\in E.
        \end{align*}
    In many applications one is interested in computing the curvature values of all edges in a graph.
    
    To compute $\kappa(x,y)$ efficiently, it is better to start from a spanning tree tailored to the edge $\{x,y\}$ rather than from an arbitrary tree. Let $H_{xy}$ be the tree on the local vertex set $\{x\}\cup N(x)\cup N(y)$ with edge set
        \begin{align*}
            E(H_{xy}) &= \bigl\{\{x,u\}: u\in N(x)\bigr\}
            \cup \bigl\{\{y,v\}: v\in N(y)\setminus(\{x\}\cup N(x))\bigr\}.
        \end{align*}
    Extend $H_{xy}$ arbitrarily to a spanning tree $T_{xy}\in\setoftrees$. In $T_{xy}$, every vertex in the support of $m_x$ is adjacent to $x$, and every vertex in the support of $m_y$ is either adjacent to $x$ or adjacent to $y$. Hence for any $u\in\supp(m_x)$ and $v\in\supp(m_y)$ we have $d_{T_{xy}}(u,v)\le 3$, and therefore
        \begin{align*}
            \wass_{T_{xy}}(m_x,m_y) &\le 3.
        \end{align*}
    We then run steepest descent on the function $T\mapsto\wass_T(m_x,m_y)$ starting from the tree $T_{xy}$. The output is a tree $T_\star$ satisfying
        \begin{align*}
            \wass_{T_\star}(m_x,m_y) &= \wass(m_x,m_y),
        \end{align*}
    and hence
        \begin{align*}
            \kappa(x,y) &= 1-\wass_{T_\star}(m_x,m_y).
        \end{align*}

    This initialization gives a strictly better convergence estimate. Every entry of $m_x-m_y$ is an integer multiple of
        \begin{align*}
            \alpha_{xy} &:= \frac{1}{\operatorname{lcm}(\deg(x),\deg(y))}.
        \end{align*}
    Therefore~\cref{lem:rational-measures-gap} gives $\delta(m_x,m_y)\ge \alpha_{xy}$. Moreover, the proof of~\cref{thm:steepest-descent-convergence} shows that the number of descent steps from an initial tree $T_0$ is bounded by
        \begin{align*}
            \frac{\wass_{T_0}(m_x,m_y)-\wass(m_x,m_y)}{\delta(m_x,m_y)}
            \le \frac{\wass_{T_0}(m_x,m_y)}{\delta(m_x,m_y)}.
        \end{align*}
    Choosing $T_0=T_{xy}$ and using $\wass_{T_{xy}}(m_x,m_y)\le 3$, we obtain
        \begin{align*}
            m_{xy} &\le \frac{3}{\delta(m_x,m_y)}
            \le 3\operatorname{lcm}(\deg(x),\deg(y))
            \le 3\Delta^2,
        \end{align*}
    where $m_{xy}$ denotes the number of steps required to compute $\wass(m_x,m_y)$ starting from $T_{xy}$, and $\Delta$ denotes the maximum degree of $G$. Thus a single edge curvature is obtained in at most $3\Delta^2$ descent moves, and the full list of edge curvature values is obtained in at most
        \begin{align*}
            3\sum_{\{x,y\}\in E}\operatorname{lcm}(\deg(x),\deg(y))
            \le 3|E|\Delta^2
        \end{align*}
    descent moves. We remark that if $G$ is $d$-regular then this bound can be improved to $3|E|d$. 

    One can search for a better neighboring tree more efficiently than by recomputing the transport cost on every adjacent tree from scratch.
    Choose arbitrary orientations of the edges of the current tree $T$. For each oriented edge $g=(u,v)$ of $T$, let
        \begin{align*}
            s_T(g) &:= \sum_{z\in S_u} \bigl(m_x(z)-m_y(z)\bigr),
        \end{align*}
    where $S_u$ denotes the connected component of $T\setminus\{\{u,v\}\}$ containing $u$. Then~\cref{lem:transport-on-trees-formula} gives $\wass_T(m_x,m_y)=\sum_{g\in E(T)} |s_T(g)|$, where $s_T(g)$ is computed using the chosen orientation of $g$. 
    These values are computed for all tree edges in linear time. Now fix a non-tree edge $f\in E\setminus T$, and let $e_1,\dots,e_k$ be the edges on the unique $T$-path joining the endpoints of $f$. Write $s_j:=s_T(e_j)$ and set $s_0:=0$. If we form the neighboring tree $T_i := T\setminus\{e_i\}\cup\{f\}$, then only the edges on the fundamental cycle change, and
        \begin{align*}
            \wass_{T_i}(m_x,m_y)
            = \wass_T(m_x,m_y) - \sum_{j=1}^k |s_j| + \sum_{j=0}^k |s_j-s_i|.
        \end{align*}
    Hence, for the fixed edge $f$, the best choice of $e_i$ is obtained by taking $s_i$ to be a median of the numbers $0,s_1,\dots,s_k$. Thus one can find the best exchange associated with $f$ in time linear in the length of its fundamental cycle, rather than by checking every deletion and recomputing the full transport cost each time.

    Summing over all non-tree edges gives a per-step search cost of
        \begin{align*}
            O\Bigl(\sum_{f\in E\setminus T} |C_T(f)|\Bigr) = O\bigl(|E||V|\bigr),
        \end{align*}
    where $C_T(f)$ denotes the fundamental cycle of $f$ with respect to $T$, and the last equality follows since each fundamental cycle has length at most $|V|$. Consequently, the initialization above yields the all-edge runtime estimate
        \begin{align*}
            O\bigl(|E|^2|V|\Delta^2\bigr),
        \end{align*}
    up to the lower-order cost of constructing the initial trees $T_{xy}$. If either $|E|$ or $\Delta$ is relatively small, our algorithm remains competitive with comparable approaches, e.g., the Hungarian algorithm, which would require $O(|E||V|^3)$ time to accomplish the same task (see, e.g.,~\cite{tomizawa1971techniques}).

\end{document}